\documentclass[letterpaper, 10 pt, conference]{ieeeconf}  
\IEEEoverridecommandlockouts                             
\usepackage{graphicx} 
\usepackage{amsmath}
\usepackage{cite}

\usepackage{amsthm}
\usepackage{mathtools}
\usepackage{changepage}
\usepackage{amssymb}
\usepackage{hyperref}
\usepackage[capitalize]{cleveref}
\usepackage[ruled,vlined]{algorithm2e}
\usepackage{algorithmic}

\usepackage{wrapfig}
\usepackage[font=small]{caption}
\usepackage{subcaption}

\usepackage{mathrsfs}

\makeatletter
\renewcommand\p@subfigure{}   
\makeatother

\newcommand{\KKTgap}{\textup{\texttt{KKT-gap}}}

\newtheorem{assumption}{Assumption}
\newtheorem{lemma}{Lemma}
\newtheorem{theorem}{Theorem}
\newtheorem{remark}{Remark}
\newtheorem{definition}{Definition}
\usepackage{xcolor}
\title{\LARGE \bf
Random-Subspace Sequential Quadratic Programming for Constrained Zeroth-Order Optimization}

\author{Runyu Zhang, Gioele Zardini
\thanks{The authors are with the Laboratory for Information and Decision Systems, Massachusetts Institute of Technology, Cambridge, MA, USA (e-mails: \{runyuzha, gzardini\}@mit.edu).}
\thanks{
Zhang was supported by the MIT Postdoctoral Fellowship
Program for Engineering Excellence; Zardini was supported by DARPA (D25AC00373).
}
}

\begin{document}

\maketitle
\thispagestyle{empty}
\pagestyle{empty}

\begin{abstract}
We study nonlinear constrained optimization problems in which only function evaluations of the objective and constraints are available. Existing zeroth-order methods rely on noisy gradient and Jacobian surrogates in high dimensions, making it difficult to simultaneously achieve computational efficiency and accurate constraint satisfaction. We propose a zeroth-order random-subspace sequential quadratic programming method (ZO-RS-SQP) that combines two-point directional estimation with low-dimensional SQP updates. At each iteration, the method samples a random low-dimensional subspace, estimates the projected objective gradient and constraint Jacobians using two-point evaluations, and solves a reduced quadratic program to compute the step. As a result, the per-iteration evaluation cost scales with the subspace dimension rather than the ambient dimension, while retaining the structured linearized-constraint treatment of SQP. We also consider an Armijo line-search variant that improves robustness in practice. Under standard smoothness and regularity assumptions, we establish convergence to first-order KKT points with high probability. Numerical experiments illustrate the effectiveness of the proposed approach on nonlinear constrained problems.
\end{abstract}
\section{Introduction}
We consider zeroth-order nonlinear constrained optimization problems of the form
\begin{equation*}
\textstyle \min_{x} f(x) \quad \text{s.t.} \quad h(x)=0,\; g(x)\le 0,
\end{equation*}
where the objective and constraints are accessible only through function evaluations.
Problems of this type arise in simulation-based optimization, engineering design, and learning-based systems, where gradients and constraint Jacobians are unavailable or prohibitively expensive to compute \cite{liu2020primer,pasupathy2018sampling}.
In such settings, optimization algorithms must infer local first-order information from function values alone.

While zeroth-order methods are now well developed for unconstrained optimization~\cite{duchi2015optimal,nesterov2017random,berahas2022theoretical,ren2023escaping}, extending them to general nonlinear constrained problems remains significantly more challenging. 
The main difficulty is that one must estimate not only the gradient of the objective, but also the Jacobians of the equality and inequality constraints. 
In high-dimensional problems, constructing these approximations in the ambient space can require a large number of function evaluations. 
As a consequence, methods that scale favorably with dimension often rely on coarse randomized estimators or penalty-based updates, which can make accurate constraint handling more difficult.

Much of the existing literature on constrained zeroth-order optimization therefore adopts Lagrangian-based primal--dual or penalty-style formulations~\cite{zhou_zeroth-order_2025,chen_model-free_2022,yi_linear_2021,nguyen_stochastic_2022,li_zeroth-order_2022,maheshwari22,liu_min-max_2020}. 
These methods are flexible and broadly applicable, but they can be sensitive to parameter tuning, may struggle to maintain feasibility throughout the optimization process, and often provide their strongest guarantees in convex settings. 
By contrast, sequential quadratic programming (SQP) is one of the most effective frameworks for smooth constrained optimization, because it explicitly linearizes the constraints and computes structured search directions by solving quadratic subproblems~\cite{gould2003galahad,liu2018comparison}. 
The challenge is that classical SQP relies on gradient and Jacobian information that is unavailable in the zeroth-order regime.

This tension motivates the following question: can one retain the structured constraint handling of SQP while reducing the cost of derivative estimation so that it scales with a low-dimensional subspace rather than the full ambient dimension? In this paper, we answer this question by proposing a zeroth-order random-subspace sequential quadratic programming method (ZO-RS-SQP).

At each iteration, ZO-RS-SQP samples a low-dimensional random subspace and restricts the search direction to that subspace.
The objective gradient and constraint Jacobians are not estimated in full. 
Instead, we compute only their projections onto the sampled subspace using two-point directional evaluations. 
These projected quantities define a reduced quadratic program, analogous to an SQP subproblem, whose solution is then lifted back to the ambient space. 
In this way, the method combines explicit linearized constraint handling with a per-iteration evaluation cost that scales with the subspace dimension. We also consider an Armijo line-search variant that improves practical robustness by coupling the computed direction with an exact-penalty merit function.

The contributions of this paper are threefold. 
First, we formulate a random-subspace SQP framework that is compatible with zeroth-order information and preserves the basic constrained structure of SQP updates. 
Second, under standard smoothness and regularity assumptions, we show that valid subspaces are accepted with positive probability and establish a high-probability $\mathcal{O}(T^{-1/2})$ bound on an averaged first-order KKT residual for the corresponding exact reduced model.
Third, we illustrate on nonlinear constrained problems that the proposed method is competitive in high-dimensional settings, and that the line-search variant typically yields more stable behavior.

The proposed approach lies at the intersection of several lines of work. 
Subspace methods have long been used to reduce the cost of large-scale optimization, and subspace variants of SQP have also been studied in first-order settings~\cite{wang2006subspace,yuan2007subspace,lee2019subspace,zhao2017subspace}.
However, these methods typically rely on deterministic subspaces and exact derivatives. 
Random-subspace techniques have also been used in unconstrained zeroth-order optimization to reduce the dimension of gradient estimation~\cite{nozawa2025zeroth,kozak2021zeroth,cartis2023scalable}, but they have not been systematically integrated with SQP-style constraint handling for nonlinear constrained problems. 
Our method bridges these directions by combining random subspace exploration, two-point zeroth-order estimation, and reduced SQP updates in a single framework.

Finally, this paper clarifies the connection to our prior work~\cite{zhang2025zeroth}, which approached related constrained zeroth-order problems from a control-theoretic perspective via feedback linearization. 
The present formulation makes the underlying SQP structure explicit, interprets the update as optimization over random low-dimensional subspaces, and thereby places that earlier viewpoint within a broader constrained optimization framework. 
To isolate the effect of subspace restriction in the current analysis, Section~\ref{sec:conv-analysis} studies the exact reduced SQP model induced by the sampled subspace; incorporating the finite-difference approximation error of the fully zeroth-order implementation is left for future work.

\section{Preliminaries}
We consider the nonlinear constrained optimization problem
\begin{equation}
\begin{aligned}
&\textstyle \min_{x \in \mathbb{R}^n} \quad f(x) \\
\text{s.t.} \quad & h(x) = 0, \quad g(x) \leq 0,
\end{aligned}
\label{eq:main_problem}
\end{equation}
where~$f : \mathbb{R}^n \to \mathbb{R}$ is the objective function,~$h : \mathbb{R}^n \to \mathbb{R}^{m_e}$ collects the equality constraints, and~$g : \mathbb{R}^n \to \mathbb{R}^{m_i}$ collects the inequality constraints. 
Throughout, we assume that~$f$,~$h$, and~$g$ are continuously differentiable, and denote by~$\nabla f(x) \in \mathbb{R}^n$,~$J_h(x) \in \mathbb{R}^{m_e \times n}$, and~$J_g(x) \in \mathbb{R}^{m_i \times n}$ the gradient and Jacobians at~$x$, respectively. We denote by $\mathrm{St}(n,d) := \{ U \in \mathbb{R}^{n\times d} : U^\top U = I_d \}$ the Stiefel manifold, and say that a random matrix $U$ is Haar-uniform on $\mathrm{St}(n,d)$ if it is distributed according to the unique probability measure invariant under left multiplication by orthogonal matrices.

For a feasible point $x$, the first-order Karush--Kuhn--Tucker (KKT) conditions assert the existence of multipliers~$\lambda \in \mathbb{R}^{m_e}$ and~$\mu \in \mathbb{R}_+^{m_i}$ such that
\begin{equation}
\begin{aligned}
&\nabla f(x) + J_h(x)^\top \lambda + J_g(x)^\top \mu = 0, \\
&h(x) = 0, \quad 
g(x) \leq 0,\\
&\mu \geq 0, \quad
\mu_i g_i(x) = 0, ~ i = 1, \dots, m_i.
\end{aligned}
\label{eq:kkt_conditions}
\end{equation}

To quantify first-order optimality, we use the KKT gap
\begin{align}
\KKTgap(x;&\lambda,\mu)
\!:=\!
\max \!\Big\{\!
\|\nabla f(x) \!+\! J_h(x)^{\!\top} \!\lambda \!+\! J_g(x)^{\!\top} \!\mu\|_2,\,\notag\\
&\|h(x)\|_\infty,\,
\|[g(x)]_+\|_\infty,\,
\|\mu \odot g(x)\|_\infty\label{eq:kkt_gap}
\Big\},
\end{align}
where~$[g(x)]_+ := [\max\{g_1(x),0\},\dots,\max\{g_{m_i}(x),0\}]^\top$, and $\odot$ denotes the Hadamard product. 
In particular, $\KKTgap(x;\lambda,\mu)=0$ if and only if $(x,\lambda,\mu)$ satisfies the KKT conditions in \eqref{eq:kkt_conditions}. 
When the multipliers are not explicitly available, one may consider the minimum residual over admissible multipliers, but for algorithmic development it is often convenient to work with \eqref{eq:kkt_gap} directly.

\subsection{Sequential Quadratic Programming}

Sequential quadratic programming (SQP) is a standard approach for solving \eqref{eq:main_problem}. At an iterate $x_t$, SQP constructs a local quadratic model of the objective together with linearized constraints, and computes a step $\Delta x_t$ by solving the subproblem
\begin{equation}
\begin{aligned}
\min_{\Delta x \in \mathbb{R}^n} \quad &
\langle \nabla f(x_t), \Delta x \rangle + \frac{L_t}{2}\|\Delta x\|_2^2 \\
\text{s.t.} \quad &
h(x_t) + J_h(x_t)\Delta x = 0, \\
&
g(x_t) + J_g(x_t)\Delta x \leq 0,
\end{aligned}
\label{eq:full_sqp}
\end{equation}
Note that, to better align with the zeroth-order setting, here we adopt a first-order (proximal) SQP variant (cf. \cite{oztoprak_constrained_2021,zhang2025constrainedoptimizationcontrolperspective}) in which the quadratic term $\frac{L_t}{2}\|\Delta x\|_2^2$ replaces the exact or approximated Hessian. This choice is motivated by the fact that first-order information can be more reliably approximated in zeroth-order regimes, making it a natural foundation for building tractable SQP-type methods .
The iterate is then updated as $x_{t+1} = x_t + \eta_t\Delta x_t$, possibly with a line search step to determine the stepsize $\eta_t$.

The main computational bottleneck in \eqref{eq:full_sqp} is that both the search direction and the linearized constraints are formed in the ambient space $\mathbb{R}^n$. This becomes expensive in high dimensions, especially when evaluating or approximating $\nabla f(x_t)$, $J_h(x_t)$, and $J_g(x_t)$ is itself costly.

\subsection{Subspace and Random-Subspace SQP}

A natural way to reduce the per-iteration complexity is to restrict the SQP step to a low-dimensional subspace. Let $U_t \in \mathbb{R}^{n \times d}$ be a matrix with orthonormal columns, where $d \ll n$, and consider search directions of the form
\begin{equation}
\textstyle \Delta x_t = U_t \alpha_t, \qquad \alpha_t \in \mathbb{R}^d.
\label{eq:subspace_param}
\end{equation}
Substituting \eqref{eq:subspace_param} into \eqref{eq:full_sqp} yields the reduced subproblem
\begin{equation}
\begin{aligned}
\min_{\alpha \in \mathbb{R}^d} \quad &
\langle U_t^\top \nabla f(x_t), \alpha \rangle + \frac{L_t}{2}\|\alpha\|_2^2 \\
\text{s.t.} \quad &
h(x_t) + J_h(x_t) U_t \alpha = 0, \\
&
g(x_t) + J_g(x_t) U_t \alpha \leq 0.
\end{aligned}
\label{eq:subspace_sqp}
\end{equation}
Since $U_t^\top U_t = I_d$, the quadratic regularization remains unchanged under this parametrization.

The matrix $U_t$ may be chosen deterministically or randomly. In this work, we focus on \emph{random-subspace SQP}, where $U_t$ is sampled Haar uniformly on $\mathrm{St}(n,d)$ (Algorithm \ref{alg:rssqp} line 4-5), typically as an orthonormal basis of a randomly generated $d$-dimensional subspace. Since the reduced subproblem may be infeasible for a given subspace, we adopt a rejection sampling strategy: if the resulting subproblem is infeasible or fails to admit bounded multipliers, the sampled subspace is discarded and resampled until a valid subspace is obtained (see Algorithm \ref{alg:rssqp} Line 12-14).

The effectiveness of \eqref{eq:subspace_sqp} depends on access to the reduced quantities
\vspace{-10pt}
\begin{equation}
U_t^\top \nabla f(x_t), \quad J_h(x_t) U_t, \quad J_g(x_t) U_t,
\label{eq:reduced_quantities}
\end{equation}
rather than the full gradient and Jacobians. This observation is particularly useful in zeroth-order settings, where direct derivative information is unavailable.

\subsection{Zeroth-Order Optimization and Two-Point Estimation}

In zeroth-order optimization, the algorithm only has access to function evaluations, rather than explicit gradients or Jacobians. A standard approach is to estimate directional derivatives through finite differences. For a unit vector $u \in \mathbb{R}^n$ and smoothing radius $r > 0$, the two-point estimator (cf. \cite{duchi2015optimal,nesterov2017random}) for the directional derivative of $f$ at $x$ along $u$ is
\begin{equation}
\frac{f(x+r u)-f(x-r u)}{2r}.
\label{eq:two_point_scalar}
\end{equation}
Under standard smoothness assumptions, \eqref{eq:two_point_scalar} provides an accurate approximation of $\langle \nabla f(x), u \rangle$.

This idea extends naturally to vector-valued constraint functions. For $h : \mathbb{R}^n \to \mathbb{R}^{m_e}$ and $g : \mathbb{R}^n \to \mathbb{R}^{m_i}$, the two-point estimators
\begin{equation}
\frac{h(x+r u)-h(x-r u)}{2r}
\quad \text{and} \quad
\frac{g(x+r u)-g(x-r u)}{2r}
\label{eq:two_point_vector}
\end{equation}
approximate $J_h(x)u$ and $J_g(x)u$, respectively.

Hence, in the subspace setting, let $U_t = [u_{t,1},\dots,u_{t,d}]$ with orthonormal columns. Applying \eqref{eq:two_point_scalar} and \eqref{eq:two_point_vector} along each subspace direction yields approximations of the reduced quantities in \eqref{eq:reduced_quantities}. More precisely, the $j$-th entry of $U_t^\top \nabla f(x_t)$ can be estimated by
\begin{equation}
\frac{f(x_t+r u_{t,j})-f(x_t-r u_{t,j})}{2r},
\label{eq:two_point_reduced_grad}
\end{equation}
and the $j$-th columns of $J_h(x_t)U_t$ and $J_g(x_t)U_t$ can be estimated by
\begin{equation}
\frac{h(x_t\!+\!r u_{t,j})\!-\!h(x_t\!-\!r u_{t,j})}{2r},~~
\frac{g(x_t\!+\!r u_{t,j})\!-\!g(x_t\!-\!r u_{t,j})}{2r}.
\label{eq:two_point_reduced_jac}
\end{equation}
Thus, the reduced SQP model can be constructed using only zeroth-order information, with a cost that scales with the subspace dimension $d$ rather than the ambient dimension $n$.

In this paper, we focus on the two-point estimator described above. It provides a simple and effective mechanism for approximating the projected gradient and Jacobian information required by the random-subspace SQP subproblem.

\section{Zeroth-Order Random-Subspace SQP (ZO-RS-SQP)}

We now describe the proposed method, which combines the subspace SQP formulation in \eqref{eq:subspace_sqp} with the two-point estimators introduced in \eqref{eq:two_point_reduced_grad}--\eqref{eq:two_point_reduced_jac}. The key idea is to construct the reduced SQP model directly from zeroth-order information, without forming full gradients or Jacobians.

At iteration $t$, let $U_t \in \mathbb{R}^{n \times d}$ be an orthonormal basis of the sampled subspace. Using the estimators defined in \eqref{eq:two_point_reduced_grad}--\eqref{eq:two_point_reduced_jac}, we form approximations
\begin{align*}
\widehat{c}_t \approx U_t^\top \nabla f(x_t), \quad
\widehat{A}_t \approx J_h(x_t) U_t, \quad
\widehat{B}_t \approx J_g(x_t) U_t.
\end{align*}
These quantities are then used to construct the reduced SQP subproblem
\begin{equation}
\begin{aligned}
\textstyle& \min_{\alpha \in \mathbb{R}^d} \quad 
\langle \widehat{c}_t, \alpha \rangle + \frac{L_t}{2}\|\alpha\|_2^2 \\
\text{s.t.} \quad& 
h(x_t) + \widehat{A}_t \alpha = 0, \quad 
g(x_t) + \widehat{B}_t \alpha \leq 0,
\end{aligned}
\label{eq:zo_sqp}
\end{equation}
and the search direction is given by
$\Delta x_t = U_t \alpha_t,$
where $\alpha_t$ is a solution of \eqref{eq:zo_sqp}. Compared to the full SQP step, the model in \eqref{eq:zo_sqp} depends only on reduced quantities and can be constructed using $2d$ function evaluations.

The subspace $U_t$ is generated by sampling a Gaussian matrix $G_t \in \mathbb{R}^{n \times d}$ with i.i.d.\ standard normal entries and computing its thin QR factorization $G_t = U_t R_t.$ The full algorithm is summarized as in Alg. \ref{alg:rssqp}.

\begin{algorithm}[htbp]
\LinesNumbered
\caption{ZO-RS-SQP}
\label{alg:rssqp}
\KwIn{Initial point $x_0 \in \mathbb{R}^n$, subspace dimension $d$, smoothing radius $r>0$, parameters $\{L_t\}$, $\{\eta_t\}$, thresholds $\Lambda, M$}
\For{$t=0,1,\dots,T-1$}{
    accepted $\leftarrow$ \texttt{false}\;
    \While{accepted = \texttt{false}}{
        Sample $G_t \in \mathbb{R}^{n\times d}$ with i.i.d.\ $\mathcal{N}(0,1)$ entries\;
        Compute the thin QR factorization $G_t = U_t R_t$\;
        \For{$j=1,\dots,d$}{
            Let $u_{t,j}$ be the $j$th column of $U_t$\;
            Evaluate $f(x_t \pm r u_{t,j})$, $h(x_t \pm r u_{t,j})$, and $g(x_t \pm r u_{t,j})$\;
            Set
            \begin{align*}
                \widehat{c}_{t,j}
                &= \frac{f(x_t + r u_{t,j}) - f(x_t - r u_{t,j})}{2r},\\
                \widehat{A}_{t,:,j}
                &= \frac{h(x_t + r u_{t,j}) - h(x_t - r u_{t,j})}{2r},\\
                \widehat{B}_{t,:,j}
                &= \frac{g(x_t + r u_{t,j}) - g(x_t - r u_{t,j})}{2r}.
            \end{align*}
        }
        Evaluate $h_t = h(x_t)$ and $g_t = g(x_t)$\;
        Attempt to solve
        \begin{equation}\label{eq:rssqp-with-dual}
        \begin{aligned}
        &\min_{\alpha \in \mathbb{R}^d}
        && \langle \widehat{c}_t, \alpha \rangle + \frac{L_t}{2}\|\alpha\|_2^2 \\
        &\text{s.t.}
        && h_t + \widehat{A}_t \alpha = 0,
        \quad \textup{dual variable } \lambda_t, \\
        &&
        & g_t + \widehat{B}_t \alpha \le 0,
        \quad \textup{dual variable } \mu_t .
        \end{aligned}
        \end{equation}
        \uIf{\eqref{eq:rssqp-with-dual} is feasible and returns $(\alpha_t,\lambda_t,\mu_t)$ with
        $\|\lambda_t\|_\infty \le \Lambda$ and $\|\mu_t\|_\infty \le M$}{
            accepted $\leftarrow$ \texttt{true}\;
        }
        \Else{
            reject the sampled subspace $U_t$ and resample from Line 4\;
        }
    }
    Set $\Delta x_t = U_t \alpha_t$\;
    Update $x_{t+1} = x_t + \eta_t \Delta x_t$\;
}
\KwOut{$x_T$}
\end{algorithm}

\subsection{Merit Function and Armijo Line Search}

To ensure stable progress toward both feasibility and optimality, we employ a merit-function-based line search under suitable smoothness assumptions (introduces in the next section). Specifically, we define the merit function
\begin{equation}
\Phi(x)
:=
f(x) + \tau \|h(x)\|_1 + \tau \|[g(x)]_+\|_1,
\label{eq:merit_function}
\end{equation}
where $\tau > 0$ is a penalty parameter. For appropriately chosen $\tau$ and stepsizes, one can show that the update direction produced by Algorithm~\ref{alg:rssqp} yields a sufficient decrease in $\Phi(x)$; this property is formalized in Theorem \ref{thm:merit_decrease} in the next section.

The existence of such a merit function motivates the use of a line-search procedure to ensure descent along the computed direction. Given a search direction $\Delta x_t$, the step size is selected via Armijo backtracking applied to $\Phi$. Let $\sigma \in (0,1)$ and $\beta \in (0,1)$ be fixed constants. Starting from an initial step size $\eta = 1$, the algorithm reduces $\eta$ geometrically until the sufficient decrease condition
\begin{equation*}
\textstyle \Phi(x_t + \eta \Delta x_t)
\le
\Phi(x_t) + \sigma \eta D_t
\label{eq:armijo_condition}
\end{equation*}
is satisfied, where $D_t$ is a directional derivative surrogate of $\Phi$ along $\Delta x_t$.

In our implementation, $D_t$ is computed using the reduced model. Let $\widehat{c}_t$, $\widehat{A}_t$, and $\widehat{B}_t$ be defined as in \eqref{eq:zo_sqp}. Writing $\Delta x_t = U_t \alpha_t$, we define
\begin{equation}
\begin{split}
\textstyle D_t
:=
\widehat{c}_t^\top \alpha_t
+ \tau \sum_{i=1}^{m_e} \operatorname{sign}(h_i(x_t)) \big(\widehat{A}_t \alpha_t\big)_i
\\\textstyle + \tau \sum_{i=1}^{m_i} \mathbf{1}\{g_i(x_t) > 0\} \big(\widehat{B}_t \alpha_t\big)_i,
\end{split}
\label{eq:directional_derivative}
\end{equation}
which approximates the directional derivative of $\Phi$ along $\Delta x_t$. The resulting line search procedure is summarized in Algorithm~\ref{alg:armijo}. 

\begin{algorithm}[htbp]
\caption{Armijo Line Search}
\label{alg:armijo}
\KwIn{Current iterate $x_t$, direction $\Delta x_t$, merit function $\Phi$, parameters $\sigma \in (0,1)$, $\beta \in (0,1)$}
Compute $D_t$ using \eqref{eq:directional_derivative}\;
Set $\eta \gets 1$\;
\If{$D_t \ge 0$}{
    \Return $\eta = 0$\;
}
\While{$\Phi(x_t + \eta \Delta x_t) > \Phi(x_t) + \sigma \eta D_t$}{
    $\eta \gets \beta \eta$\;
}
\Return $\eta$\;
\end{algorithm}

\section{Convergence Analysis}
\label{sec:conv-analysis}

In this section, we establish convergence properties of the proposed method. Under standard smoothness assumptions, two-point zeroth-order estimators can approximate directional derivatives arbitrarily well given sufficiently accurate function evaluations (cf.~Eq.~(24) in \cite{duchi2015optimal}). Hence, to facilitate the analysis while retaining the essential structure of the method, we consider the exact subproblem~\eqref{eq:subspace_sqp} and defer a treatment of approximation error to future work. We begin by stating the standing assumptions.
\begin{assumption}[Smoothness]
\label{ass:smoothness}
The functions $f$, $h_i$, and $g_i$ are continuously differentiable on $\mathbb{R}^n$. Moreover, their gradients are Lipschitz continuous, with constants $\ell_f$, $\ell_{h,i}$, and $\ell_{g,i}$, respectively, on an open set containing the level set
\begin{equation}
\mathcal{L} := \{ x \in \mathbb{R}^n : \Phi(x) \le \Phi(x_0) \},
\end{equation}
where $\Phi$ is defined in \eqref{eq:merit_function}. We also assume that $\mathcal{L}$ is compact.
\end{assumption}

\begin{assumption}[Bounded constraint gradients]
\label{ass:bounded_grad}
There exist constants $H_f, H_h, H_g > 0$ such that for all $x \in \mathcal{L}$,
\begin{align*}
\|\nabla f(x)\|_2 &\le H_f,\\
\|\nabla h_i(x)\|_2 &\le H_h,
\qquad i = 1,\dots,m_e,\\
\|\nabla g_i(x)\|_2 &\le H_g,
\qquad i = 1,\dots,m_i.
\end{align*}
\end{assumption}

\begin{definition}[\small Mangasarian--Fromovitz constraint qualification]
\label{def:mfcq}
Let $A \in \mathbb{R}^{m_e \times n}$ and $B \in \mathbb{R}^{m_i \times n}$.
Given constants $\sigma>0$, $\gamma>0$, and $R>0$, we say that the pair
$(A,B)$ satisfies $\mathrm{MFCQ}(\sigma,\gamma,R)$ if
\begin{enumerate}
    \item $\sigma_{\min}(A)\ge \sigma;$
    \item there exists a direction $d\in\mathbb R^n$ such that
    \[
    \|d\|_2\le R,
    \qquad
    Ad=0,
    \qquad
    Bd\le -\gamma \mathbf 1.
    \]
\end{enumerate}
\end{definition}

\begin{assumption}[Uniform MFCQ]
\label{ass:orig_regularity}
There exist constants $\sigma_*>0$, $\gamma>0$, and $R_*>0$ such that, for every
$x\in\mathcal L$, the pair $(J_h(x),J_g(x))$ satisfies
$\mathrm{MFCQ}(\sigma_*,\gamma,R_*)$.
\end{assumption}

\subsection{Acceptance Rate}

\begin{theorem}[Positive acceptance probability]
\label{thm:pacc_full_levelset}
Suppose Assumptions~\ref{ass:smoothness}--\ref{ass:orig_regularity} hold, and let $d \ge m_e+1$. At iteration $t$, let $U_t$ be a i.i.d. sample generated from Algorithm \ref{alg:rssqp}, and define the event
\begin{align*}
\mathcal A_t
\!:=\!
\Big\{\!
\text{Eq \eqref{eq:subspace_sqp} is feasible  with }
\|\lambda_t\|_\infty \!\!\le \!\Lambda,
\!\|\mu_t\|_\infty \!\!\le\! M
\!\Big\}\!.
\end{align*}
Then there exists constant $\Lambda^*, M^*$ and $p_{\mathrm{acc}} >0$ (which only depends on the parameters in Assumption \ref{ass:smoothness}--\ref{ass:orig_regularity}) such that for any $\Lambda, M
    \ge
    \Lambda^*, M^*$ we have $\mathbb P\!\left(
\mathcal A_t 
\right) \ge p_\mathrm{acc}$.
\end{theorem}

\begin{proof}
Fix $x \in \mathcal L$. By Lemma~\ref{lem:robust_reduced_regularity}, if the sampled subspace lies within principal-angle distance at most $\theta_*$ of $\mathcal S(x)$, then the reduced subproblem is feasible and admits multipliers bounded by $\Lambda_*$ and $M_*$. Therefore,
\begin{align*}
\textstyle \mathbb P\!\left(
\mathcal A_t
\right)
\ge
\mathbb P\!\left(
\angle_{\max}\big(\operatorname{span}(\widetilde U_t),\mathcal S(x_t)\big) \le \theta_*
\;\middle|\;
\mathscr F_t
\right)
\end{align*}
whenever $\Lambda \ge \Lambda_*$ and $M \ge M_*$. Since $ U_t$ is Haar-uniform and $\mathcal L$ is compact, the right-hand side admits a uniform positive lower bound $p_\mathrm{acc} > 0$. 
\end{proof}

\subsection{Descent of the Merit Function}

We analyze the one-sided directional derivative of the merit function \eqref{eq:merit_function}. In particular, we consider the merit function with $\tau \ge \Lambda, M$.

Note that in Algorithm \ref{alg:rssqp} and Eq. \eqref{eq:subspace_sqp} the KKT conditions hold for $\alpha_t$ and multipliers $(\lambda_t,\mu_t)$: 
\begin{equation}
\begin{split}
U_t^{\!\top} \!\nabla\! f(x_t) \!+\! L_t \alpha_t \!+\! U_t^{\!\top}\! J_h(x_t)^{\!\top} \!\lambda_t \!+\!  U_t^{\!\top} \!J_g(x_t)^{\!\top} \! \mu_t &= 0, \\
h(x_t) + J_h(x_t)U_t \alpha_t &= 0, \\
g(x_t) + J_g(x_t)U_t \alpha_t &\le 0, \\
\mu_t \ge 0,\quad  
(\mu_t)\odot \big( g(x_t) + J_g(x_t)U_t \alpha_t \big) &= 0.
\end{split}
\label{eq:reduced-KKT}
\end{equation}

\begin{theorem}[Decrease of Merit Function per-step]
\label{thm:merit_decrease} Under Assumption \ref{ass:smoothness}, and
  define
\begin{align}
\textstyle C_\Phi
:=
\frac{\ell_f}{2}
+
\frac{\tau}{2}\sum_{i=1}^{m_e} \ell_{h,i}
+
\frac{\tau}{2}\sum_{i=1}^{m_i} \ell_{g,i}.\label{eq:def-CPhi}
\end{align}
Then by running Algorithm \ref{alg:rssqp} we get 
\begin{equation}
\begin{split}
\Phi(x_{t+1}) - \Phi(x_t)
&\le
-\big(L_t - C_\Phi\big)\|\Delta x_t\|_2^2
\label{eq:merit_decrease_main}
\end{split}
\end{equation}
\end{theorem}

\begin{proof}
We estimate the three terms in the merit function separately.

From Assumption \ref{ass:smoothness} we have
\begin{align*}
\textstyle f(x_t+\Delta x_t)
\le
f(x_t) + \nabla f(x_t)^\top \Delta x_t + \frac{\ell_f}{2}\|\Delta x_t\|_2^2,
\end{align*}
and
\begin{align*}
h_i(x_t+\Delta x_t)
=
h_i(x_t) + \nabla h_i(x_t)^\top \Delta x_t + \rho_{h,i,t},\\
\textstyle |\rho_{h,i,t}| \le \frac{\ell_{h,i}}{2}\|\Delta x_t\|_2^2.
\end{align*}
Since $\Delta x_t = U_t \alpha_t$, the linearized equality feasibility condition gives
\begin{align*}
h(x_t) + J_h(x_t)\Delta x_t
=
h(x_t) + J_h(x_t)U_t\alpha_t
=
0.
\end{align*}
Hence
$|h_i(x_t+\Delta x_t)|
\le
\frac{\ell_{h,i}}{2}\|\Delta x_t\|_2^2.$
Summing over $i$ yields
\begin{align*}
\|h(x_t\!+\!\Delta x_t)\|_1
\!-\!
\|h(x_t)\|_1
\!\le\!
\!-\!\|h(x_t)\|_1
\!+\!
\frac{1}{2}\!\sum_{i=1}^{m_e}\!\ell_{h,i}\|\Delta x_t\|_2^2.
\end{align*}

Similarly, for each inequality constraint component,
\begin{align*}
g_i(x_t+\Delta x_t)
=
g_i(x_t) + \nabla g_i(x_t)^\top \Delta x_t + \rho_{g,i,t},\\
|\rho_{g,i,t}| \le \frac{\ell_{g,i}}{2}\|\Delta x_t\|_2^2.
\end{align*}
And the linearized inequality feasibility condition gives
\begin{align*}
g(x_t) + J_g(x_t)\Delta x_t
=
g(x_t) + J_g(x_t)U_t \alpha_t
\le 0.
\end{align*}
Thus for each $i$,
\begin{align*}
 [g_i(x_t+\Delta x_t)]_+
&=
\big[g_i(x_t) + \nabla g_i(x_t)^\top \Delta x_t + \rho_{g,i,t}\big]_+ \\
&\le
[\rho_{g,i,t}]_+ 
\le
|\rho_{g,i,t}| 
\le
\frac{\ell_{g,i}}{2}\|\Delta x_t\|_2^2.
\end{align*}
Therefore
\begin{align*}
&\quad \|[g(x_t+\Delta x_t)]_+\|_1
-
\|[g(x_t)]_+\|_1
\\&\textstyle \le
-\|[g(x_t)]_+\|_1
+
\frac{1}{2}\sum_{i=1}^{m_i}\ell_{g,i}\|\Delta x_t\|_2^2.
\end{align*}

Combining the three estimates gives
\begin{align*}
\textstyle\Phi(x_{t+1}) - \Phi(x_t)
\le\;&\textstyle
\nabla f(x_t)^\top \Delta x_t
+
\frac{\ell_f}{2}\|\Delta x_t\|_2^2 \\
&\textstyle\;
-\tau \|h(x_t)\|_1
+
\frac{\tau}{2}\sum_{i=1}^{m_e}\ell_{h,i}\|\Delta x_t\|_2^2 \\
&\textstyle\;
-\tau \|[g(x_t)]_+\|_1
+
\frac{\tau}{2}\sum_{i=1}^{m_i}\ell_{g,i}\|\Delta x_t\|_2^2\\
= \nabla f(x_t)^\top \Delta x_t
-\tau  &\|h(x_t)\|_1
-\tau \|[g(x_t)]_+\|_1
+
C_\Phi \|\Delta x_t\|_2^2.
\end{align*}
We would also like to note that the above argument provide a proof that there always exists an nonnegative stepsize for Armijo linesearch (Algorithm \ref{alg:armijo}).

We now use the reduced KKT system to expand the first-order term. Since 
$\nabla f(x_t)^\top \Delta x_t
=
\nabla f(x_t)^\top U_t \alpha_t.$
Multiplying the reduced stationarity condition (first equation in \eqref{eq:reduced-KKT}) by $\alpha_t$ gives
\begin{align*}
\nabla f(x_t)^{\!\top} \!U_t \alpha_t
\!=\!
- \!L_t \|\alpha_t\|_2^2
\!-\! \lambda_t^\top J_h(x_t) U_t \alpha_t
\!-\! \mu_t^\top J_g(x_t) U_t\alpha_t.
\end{align*}
Since
$J_h(x_t) U_t \alpha_t = - h(x_t),~~
\mu_t^\top J_g(x_t) U_t \alpha_t
=
- \mu_t^\top g(x_t),$
we obtain
\begin{align*}
\nabla f(x_t)^\top U_t \alpha_t
=
- L_t \|\alpha_t\|_2^2
+ \lambda_t^\top h(x_t)
+ \mu_t^\top g(x_t).
\end{align*}
Substituting this into the previous estimate yields
\begin{align*}
\Phi(x_{t+1}) - \Phi(x_t)
\le\;&
- L_t \|\alpha_t\|_2^2
+ \lambda_t^\top h(x_t)
+ \mu_t^\top g(x_t) \\
-\tau &\|h(x_t)\|_1 
\!-\!\tau \|[g(x_t)]_+\|_1
\!+
C_\Phi \|\Delta x_t\|_2^2.
\end{align*}

Next, Hölder's inequality and the multiplier bounds imply
\begin{align*}
\lambda_t^\top h(x_t)
\le
\|\lambda_t\|_\infty \|h(x_t)\|_1
\le
\Lambda \|h(x_t)\|_1,
\end{align*}
and, since $\mu_t \ge 0$,
\begin{align*}
\mu_t^\top g(x_t)
\le
\|\mu_t\|_\infty \|[g(x_t)]_+\|_1
\le
M \|[g(x_t)]_+\|_1.
\end{align*}
Hence
\begin{align*}
\Phi(x_{t+1}) - \Phi(x_t)
\le\;&
- L_t \|\alpha_t\|_2^2
- (\tau-\Lambda)\|h(x_t)\|_1\\
&- (\tau-M)\|[g(x_t)]_+\|_1 
+ C_\Phi \|\Delta x_t\|_2^2.
\end{align*}

Finally, since $U_t$ has orthonormal columns we have $\|\Delta x_t\|_2 = \|\alpha_t\|_2.$
Therefore
\begin{align*}
 \Phi(x_{t+1}) - \Phi(x_t)
&\le
-\big(L_t - C_\Phi\big)\|\Delta x_t\|_2^2.
\end{align*}
which is exactly the claimed estimate.
\end{proof}

\begin{theorem}[One-step residual bounds]
\label{thm:kkt_gap_bound}
Under Assumptions~\ref{ass:smoothness},\ref{ass:bounded_grad}, let
$r_t := \nabla f(x_t) + J_h(x_t)^\top \lambda_t + J_g(x_t)^\top \mu_t$
Then
\begin{align}
\|U_t^\top r_t\|_2 &\le L_t \|\Delta x_t\|_2, \label{eq:Ut_rt_bound_new} \\
\|h(x_t)\|_\infty &\le H_h \|\Delta x_t\|_2, \label{eq:h_bound_dx_new} \\
\|[g(x_t)]_+\|_\infty &\le H_g \|\Delta x_t\|_2, \label{eq:gplus_bound_dx_new} \\
\|\mu_t \odot g(x_t)\|_\infty &\le M H_g \|\Delta x_t\|_2.\label{eq:comp_bound_dx_new}
\end{align}
\end{theorem}

\begin{proof}
The projected stationarity bound follows directly from the reduced stationarity equation:
\begin{align*}
U_t^\top r_t + L_t \alpha_t = 0
\Longrightarrow
\|U_t^\top r_t\|_2
=
L_t \|\alpha_t\|_2
=
L_t \|\Delta x_t\|_2.
\end{align*}
Next, from the reduced equality feasibility condition,
\begin{align*}
&h(x_t) + J_h(x_t)\Delta x_t = 0\Longrightarrow\\
&|h_i(x_t)|
\!=\!
|\!\nabla h_i(x_t)^{\!\top} \!\Delta x_t|
\!\le\!
\|\!\nabla h_i(x_t)\|_2 \|\!\Delta x_t\!\|_2 \!\le\! H_h\! \|\!\Delta x_t\!\|_2\\
&\Longrightarrow~~ \|h(x_t)\|_\infty \le H_h \|\Delta x_t\|_2.
\end{align*}
Similarly, 
\vspace{-10pt}
\begin{align*}
&\quad g(x_t) + J_g(x_t)\Delta x_t \le 0 \Longrightarrow\\
&[g_i(x_t)]_+
\!\!\le\!
|\!\nabla\! g_i(x_t)^{\!\top}\! \Delta x_t|
\!\le\!
\|\!\nabla\! g_i(x_t)\|_2 \|\!\Delta x_t\!\|_2\!\le\! H_g \!\|\!\Delta x_t\!\|_2\\
&\Longrightarrow~~
\|[g(x_t)]_+\|_\infty \le H_g \|\Delta x_t\|_2.
\end{align*}

Finally, by complementarity of the reduced subproblem,
\begin{align*}
(\mu_t)_i g_i(x_t)
=
- &(\mu_t)_i \nabla g_i(x_t)^\top \Delta x_t,
\qquad i=1,\dots,m_i,\\
\Longrightarrow
\|\mu_t \!\odot\! g(x_t)\|_\infty
&\textstyle\le\!
\|\mu_t\|_\infty
\max_{1 \le i \le m_i}
\|\nabla g_i(x_t)\|_2
\|\Delta x_t\|_2 \\
&\le
M H_g \|\Delta x_t\|_2.
\end{align*}
This proves the result.
\end{proof}

\vspace{-5pt}

\subsection{Convergence Rate}

\begin{theorem}[High-probability $O(T^{-1/2})$ convergence]
\label{thm:hp_kkt_rate}
Suppose the $\Lambda, M$ in Algorithm \ref{alg:rssqp} satisfies $\Lambda\ge\Lambda^\star, M \ge M^\star$ as in Theorem \ref{thm:pacc_full_levelset}, and $\tau \ge \Lambda, M$ for the merit function \eqref{eq:merit_function}, and suppose Assumption \ref{ass:smoothness}--\ref{ass:orig_regularity} hold. Let
\begin{align*}
\textstyle c := \inf_t L_t - C_\Phi > 0,
\qquad
\Phi_* := \inf_{x \in \mathcal L} \Phi(x),
\end{align*}
where $C_\Phi$ is defined in \eqref{eq:def-CPhi}. Then for any $\rho \!\in\! (0,1)$, if the subspace dimension $d
\ge
\frac{32}{\rho^2}
\left[
m\log\!\left(\frac{12n}{\rho m}\right)
+
\log\!\left(\frac{T}{\delta\,p_{\mathrm{acc}}}\right)
\right]$, where $m=m_i+m_e+1$ then with probability at least
$1 - \delta,$ 
\begin{align*}
\textstyle \frac{1}{T}\sum_{t=1}^T \KKTgap(x_{t};\lambda_{t},\mu_{t})
\le
C_\rho
\sqrt{\frac{\Phi(x_0)-\Phi_*}{c\,T}},
\end{align*}
where
$\textstyle C_\rho
=
\max\left\{
\sup_t L_t \sqrt{\frac{n}{(1-\rho)d}},
\,
H_h,
\,
H_g,
\,
M H_g
\right\}.$
\end{theorem}

\begin{proof}
From Theorem \ref{thm:merit_decrease},
\begin{align}\label{eq:telescoping}
\textstyle \sum_{t=0}^{T-1} \|\Delta x_t\|_2^2
\le
\frac{\Phi(x_0)-\Phi_*}{c}.
\end{align}
Since $r_t$ in Theorem \ref{thm:pacc_full_levelset} satisfies $r_t \in \operatorname{span}(\nabla f(x_t), J_h(x_t), J_g(x_t))$, applying Lemma \ref{lem:conditional_subspace_projection} we get with probability at least $1-\frac{\delta}{T}$, $\|r_t\|\le\sqrt{\frac{n}{(1-\rho)d}} \|U_tr_t\|$
By Theorem \ref{thm:pacc_full_levelset}, Lemma~\ref{lem:conditional_subspace_projection} and Theorem \ref{thm:kkt_gap_bound} (Eq. \eqref{eq:Ut_rt_bound_new}),
\begin{align*}
\textstyle \|r_t\|_2
\le
 L_t \sqrt{\frac{n}{(1-\rho)d}} \|\Delta x_t\|_2
\end{align*}
holds uniformly over $t=0,\dots,T-1$ with probability at least
$1 - \frac{T}{p_{\mathrm{acc}}}\exp\!\left(-\frac{\rho^2 d}{8}\right).$
Combining with \eqref{eq:h_bound_dx_new} - \eqref{eq:comp_bound_dx_new} yields
\begin{align*}
\textstyle \KKTgap(x_t;\lambda_t,\mu_t)
\le
C_\rho \|\Delta x_t\|_2.
\end{align*}
Combining with Eq \eqref{eq:telescoping} completes the proof.
\end{proof}
We note that the definition of the KKT gap explicitly incorporates constraint violation terms. Therefore, convergence of the KKT gap directly implies convergence to feasible region. 
\begin{remark}[On the subspace dimension $d$]
Note that $d$ scales with the dimension of the constraints $m\!=\!m_i \!+\! m_e \!+\!1$ rather than the ambient dimension $n$. Consequently, the per-iteration reduced subproblem remains low-dimensional whenever $m$ is small. This makes the proposed approach particularly appealing in regimes where the set of active constraints has relatively low effective dimension compared with the ambient problem dimension.
\end{remark}

\section{Numerical Experiments}
\subsection{Nonlinear Programming}

We first evaluate the proposed method on a synthetic nonlinear programming problem designed to capture key challenges arising in constrained optimization. Specifically, we consider a high-dimensional problem of the form
\begin{align*}
\textstyle \min_{x \in \mathbb{R}^n} \; \tfrac{1}{2} x^\top Q x + p^\top x,
\end{align*}
where $n = 100, Q \in \mathbb{R}^{n \times n}$ is positive definite and \(p \in \mathbb{R}^n\). The problem is subject to a nonlinear equality constraint and nonlinear inequality constraints:
\begin{align*}
\textstyle \sum_{i=1}^n x_i \!+\! \frac{1}{10} \sum_{i=1}^n x_i^3 \!=\! 1, 
\quad
x_i^2 \le 0.5, ~ i \!=\! 1,\! \dots, \!m_{\mathrm{ineq}}.
\end{align*}
The equality constraint introduces nonconvexity through the cubic term, while the inequality constraints impose elementwise bounds on a subset of decision variables. This combination yields a nonconvex feasible region with both equality and inequality constraints.

\begin{figure}[tb]
\centering\includegraphics[width=0.99\linewidth]{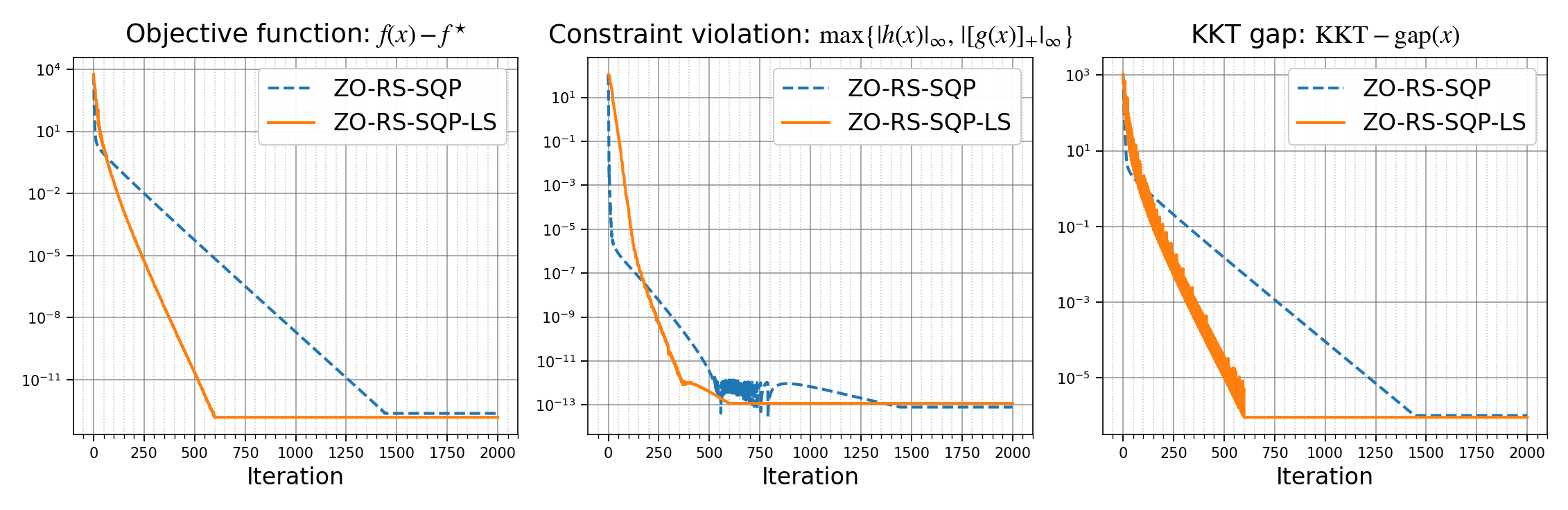}
\vspace{-5pt}
\caption{Learning curves over iterations: (left) objective value, (middle) constraint violation, and (right) KKT gap.}
    \label{fig:nonlinear-programming}
    \vspace{-5pt}
\end{figure}
Both ZO-RS-SQP and its line-search variant (ZO-RS-SQP-LS) exhibit consistent convergence in objective value, constraint violation, and KKT residual, as shown in Fig. \ref{fig:nonlinear-programming}. In particular, both methods achieve high-accuracy solutions with vanishing constraint violation, indicating that ZO-RS-SQP is able to reliably solve the nonlinear program even without line search. However, ZO-RS-SQP-LS converges slightly faster and more smoothly across iterations.

\subsection{Transient Dynamics Optimization Under Angle Separation Constraints}
We further illustrate the proposed ZO-RS-SQP method on a dynamical system inspired by transient stability in power networks. 
We consider a networked dynamical system defined over an undirected graph $\mathcal{G} = (\mathcal{V}, \mathcal{E})$ with $n = |\mathcal{V}|$ buses. The network is constructed as a ring topology augmented with second-neighbor (chordal) connections, resulting in a sparse but well-connected graph.

The set of buses is partitioned into generator nodes $\mathcal{V}_g$ and load nodes $\mathcal{V}_\ell$, with $\mathcal{V} = \mathcal{V}_g \cup \mathcal{V}_\ell$ and $\mathcal{V}_g \cap \mathcal{V}_\ell = \emptyset$. For each node $i$, we denote by $\theta_i$ the phase angle, and for generator nodes $i \in \mathcal{V}_g$ we additionally introduce the frequency deviation $\omega_i$. The coupling strengths are encoded by a symmetric matrix $B \in \mathbb{R}^{n \times n}$ that respects the graph structure, in the sense that $B_{ij} = 0$ if $(i,j) \notin \mathcal{E}$. 
The system dynamics follow a structure-preserving network model:
\begin{equation*}
\begin{aligned}
\textstyle \dot{\theta}_i &= \omega_i, \quad && i \in \mathcal{V}_g, \\
\textstyle M_g \dot{\omega}_i &\textstyle= x_i - D_g \omega_i - \sum_{j=1}^n B_{ij} \sin(\theta_i - \theta_j), \quad && i \in \mathcal{V}_g, \\
\textstyle \dot{\theta}_i &\textstyle= \frac{-d_i - \sum_{j=1}^n B_{ij} \sin(\theta_i - \theta_j)}{D_\ell}, \quad && i \in \mathcal{V}_\ell,
\end{aligned}
\end{equation*}
where $M_g$, $D_g$, and $D_\ell$ are inertia and damping parameters. The electrical power flow is given by the sinusoidal coupling term $\sum_{j} B_{ij} \sin(\theta_i - \theta_j)$. The vector $x \in \mathbb{R}^{|\mathcal{V}_g|}$ denotes the controllable generation setpoints (optimization variables), and $d_i$ are fixed demands at load buses. This model corresponds to a second-order Kuramoto-type system and is closely related to structure-preserving formulations of the classical swing equations used in transient stability analysis~\cite{dorfler2012synchronization,chiang2011direct}.%
\footnote{We emphasize that this model is a simplified abstraction and does not capture all aspects of large-scale power system dynamics. Nevertheless, it retains the essential nonlinear coupling and oscillatory behavior relevant to transient stability, and serves as a testbed for our methodology.}

To emulate a disturbance, we consider a two-stage simulation. During the fault-on period $t \in [0, t_{\mathrm{clear}}]$, the coupling matrix is modified to $B^{\mathrm{fault}}$, obtained by scaling selected edge weights to model a temporary weakening of transmission lines. After clearing, the system evolves under the nominal coupling $B$.

We consider the problem of selecting generation setpoints $x$ to minimize a quadratic cost while enforcing transient angle separation constraints:
\begin{align*}
\textstyle \min_{x \in \mathbb{R}^{|\mathcal{V}_g|}} \quad
& \textstyle \sum_{i \in \mathcal{V}_g} \left( a_i x_i^2 + b_i x_i \right) \\
\text{s.t.} \quad
& \textstyle \sum_{i \in \mathcal{V}_g} x_i = \sum_{i \in \mathcal{V}_\ell} d_i, \\
&\textstyle x_i^{\mathrm{min}} \le x_i \le x_i^{\mathrm{max}}, \quad i \in \mathcal{V}_g, \\
\textstyle \max_{t \in [0,T]}&\textstyle  \max_{\{i,j\} \in \mathcal{E}} 
|\theta_i(t;x) - \theta_j(t;x)| \le \delta_{\max}.
\end{align*}

The last constraint enforces a bound on relative phase differences over the entire trajectory and serves as a surrogate for transient stability, as excessive angle separation is associated with loss of synchronism. The trajectories $\theta(t;x)$ are obtained through nonlinear simulation, and no gradient information with respect to $x$ is assumed available, motivating the use of zeroth-order optimization methods.

\begin{figure}[tb]
\centering\includegraphics[width=0.99\linewidth]{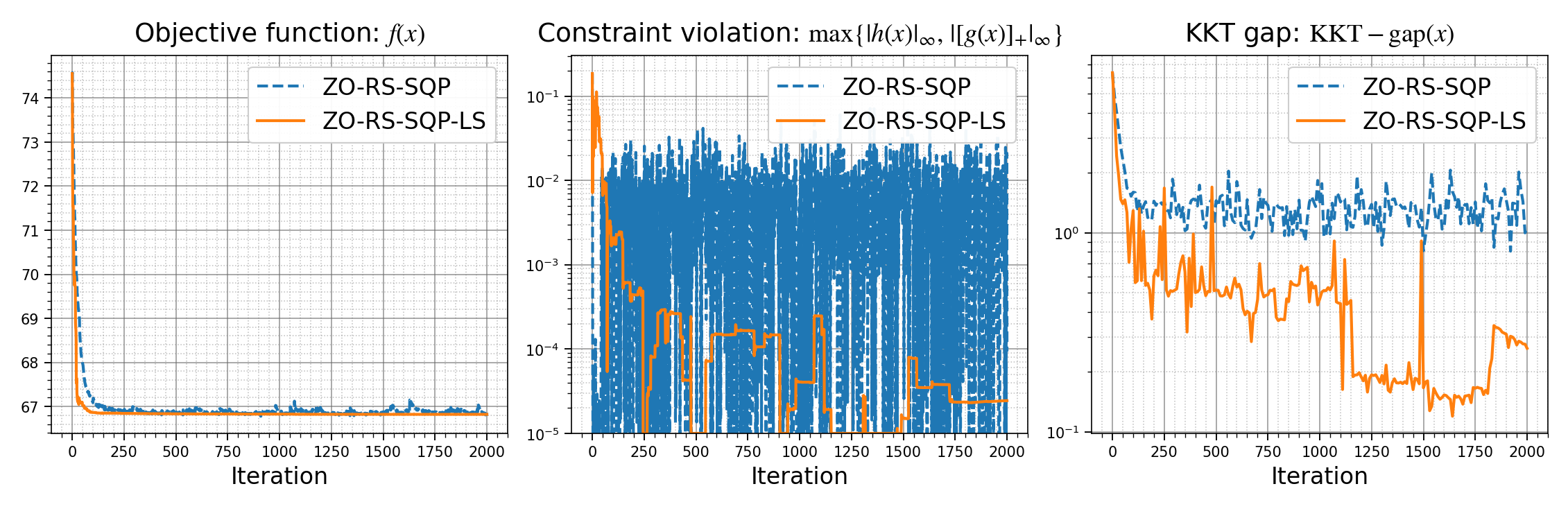}
\vspace{-5pt}
\caption{Learning curves over iterations: (left) objective value, (middle) constraint violation, and (right) KKT gap.}
    \label{fig:convergence}
    \vspace{-5pt}
\end{figure}
As shown in Fig.~\ref{fig:convergence}, both ZO-RS-SQP and ZO-RS-SQP-LS methods achieve similar objective values. While ZO-RS-SQP exhibits larger constraint violations and KKT gaps compared to ZO-RS-SQP-LS, it nevertheless attains reasonable feasibility and solution quality. In contrast, ZO-RS-SQP-LS consistently improves constraint satisfaction and reduces the KKT gap, indicating more stable convergence. These results suggest that line search effectively stabilizes zeroth-order updates by better balancing descent and feasibility. Figure~\ref{fig:theta-dynamics} further illustrates the resulting transient dynamics. The angle separation constraint is active, as the baseline solution (solving a QP ignoring the angular constraint) exhibits larger angle excursions, whereas the ZO-RS-S solution keeps the trajectories within a tighter range.

\begin{figure}[htbp]
    \centering
    \includegraphics[width=0.9\linewidth]{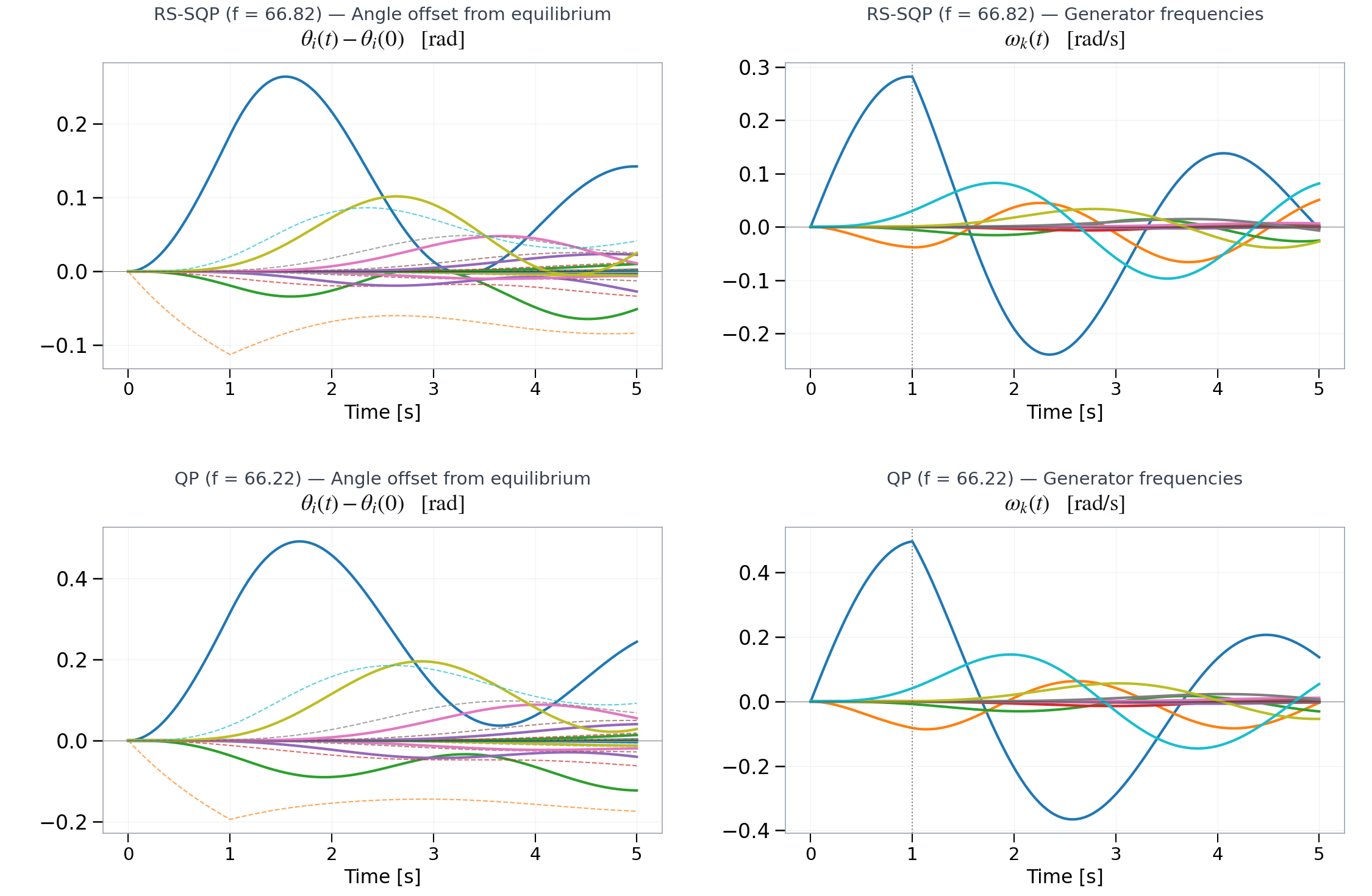}
    \vspace{-5pt}
    \caption{Transient phase (offset) and frequency dynamics under a fault disturbance. Top: solution obtained ZO-RS-S-LS. Bottom: baseline QP solution ignoring angle separation constraints. }
    \label{fig:theta-dynamics}
    \vspace{-5pt}
\end{figure}
\section{Conclusion and Future Work}

We proposed ZO-RS-SQP, a zeroth-order random-subspace SQP method for nonlinear constrained optimization. By combining subspace dimension reduction with SQP updates, the method reduces per-iteration cost while retaining explicit constraint handling. We proved convergence to first-order KKT points under standard assumptions, together with high-probability guarantees. Numerical results show that ZO-RS-SQP performs competitively and maintains strong constraint satisfaction, with line search improving stability. Future work includes incorporating approximation error into the analysis, developing adaptive subspace strategies, and extending the approach to large-scale learning and control applications.
\vspace{-10pt}
\section*{Acknowledgement}
The authors would like to thank Prof. Na Li , Prof. Jeff Shamma, and Dr. Yiqi Tian for the insightful discussions.
\bibliographystyle{ieeetr}
\bibliography{bib, references}

\vspace{-10pt}
\appendix

\begin{lemma}
\label{lem:robust_reduced_regularity}
Suppose Assumptions~\ref{ass:smoothness}--\ref{ass:orig_regularity} hold. Let
\begin{align*}
\mathcal S(x)
:=
\operatorname{span}\!\Big(
d_*(x),\;
\operatorname{range}(J_h(x)^\top)
\Big).
\end{align*}
Then there exists an angle $\theta_* \in (0,\pi/2)$, depending only on $\gamma$, $R_*$, $\sigma_*$, $H_h$, and $H_g$, such that the following holds: for every $x \in \mathcal L$ and every orthonormal matrix $U \in \mathrm{St}(n,d)$ with $d \ge m_e+1$, if the largest principal angle between $\operatorname{span}(U)$ and $\mathcal S(x)$ is at most $\theta_*$, then there exists $\Lambda^\star, M^\star$ such that the corresponding KKT multiplier pair $(\lambda,\mu)$ satisfies
    \begin{align*}
    \|\lambda\|_\infty \le \Lambda_*,
    \qquad
    \|\mu\|_\infty \le M_*.
    \end{align*}
\end{lemma}
\begin{proof}
Fix $x \in \mathcal L$ and abbreviate $d_* := d_*(x)$ and $\mathcal S := \mathcal S(x)$. Since $\mathcal S$ is generated by $d_*$ and the columns of $J_h(x)^\top$, we have $\dim \mathcal S \le m_e+1$.

Now suppose the largest principal angle between $\operatorname{span}(U)$ and $\mathcal S$ is at most $\theta$. Let $P_U$ denote the orthogonal projector onto $\operatorname{span}(U)$. By standard principal-angle geometry, for every $z \in \mathcal S$,
$\|P_U z - z\|_2 \le \sin(\theta)\|z\|_2.$
In particular,
\begin{align*}
\|P_U d_* - d_*\|_2 \le \sin(\theta)\|d_*\|_2 \le \sin(\theta)R_*.
\end{align*}

We first establish reduced feasibility. Define $\widetilde d := P_U d_* \in \operatorname{span}(U)$. Since $J_h(x)$ is bounded on $\mathcal L$ by Assumption~\ref{ass:bounded_grad},
\begin{align*}
\|J_h(x)(\widetilde d - d_*)\|_\infty
\le
H_h \|\widetilde d-d_*\|_2
\le
H_h R_* \sin(\theta).
\end{align*}
Because $h(x)+J_h(x)d_*=0$, it follows that
\begin{align*}
\|h(x)+J_h(x)\widetilde d\|_\infty
\le
H_h R_* \sin(\theta).
\end{align*}
Similarly,
\begin{align*}
g(x)+J_g(x)\widetilde d
&=
g(x)+J_g(x)d_* + J_g(x)(\widetilde d-d_*) \\
&\le
-\gamma \mathbf 1 + H_g R_* \sin(\theta)\mathbf 1.
\end{align*}
Hence, if
$\sin(\theta) \le \frac{\gamma}{4H_g R_*},
$
then
$g(x)+J_g(x)\widetilde d \le -\frac{\gamma}{2}\mathbf 1.$

The equality residual in $\widetilde d$ can be corrected within $\operatorname{span}(U)$. Consider the restricted linear map $A_U := J_h(x)\big|_{\operatorname{span}(U)} : \operatorname{span}(U) \to \mathbb{R}^m.$ Since $\operatorname{span}(U)$ is $\theta$-close to $\mathcal S$ and $J_h(x)$ has minimum singular value at least $\sigma_*$ on $\mathcal S$, a standard subspace perturbation argument implies that
$\sigma_{\min}(A_U) \ge \sigma_*/2$
for $\theta$ sufficiently small. Hence $A_U$ is surjective and admits a right inverse with operator norm at most $2/\sigma_*$. Therefore, there exists $\delta d \in \operatorname{span}(U)$ such that
\begin{align*}
J_h(x)(\widetilde d + \delta d) &= -h(x), \\
\|\delta d\|_2 &\le C \|h(x) + J_h(x)\widetilde d\|_2,\\
\Longrightarrow\quad 
\|\delta d\|_2 &\le C H_h R_* \sin(\theta).
\end{align*}

for some constant $C$ depending only on $\sigma_*$. Choosing $\theta$ small enough, we may ensure $\|\delta d\|_2$ small enough so that for $^{\mathrm{ref}} := \widetilde d + \delta d$ we still have
$\textstyle g(x)+J_g(x)d^{\mathrm{ref}} \le -\frac{\gamma}{4}\mathbf 1.$

The above statement shows that the reduced subproblem \eqref{eq:subspace_sqp} also satisfies MFCQ (Definition \ref{def:mfcq}) uniformly on $x\in\mathcal{L}$, and thus obtains uniformly bounded dual variable (cf. \cite{nocedal_numerical_2006}).
\end{proof}

\begin{lemma}
\label{lem:conditional_subspace_projection}
Let $V \in \mathrm{St}(n,m)$ be $\mathcal F_t$-measurable, and let $U_t$ be obtained by rejection sampling from Haar-uniform draws on $\mathrm{St}(n,d)$ until $\mathcal A_t$ occurs, with acceptance rate at least $p_{\mathrm{acc}}$. Then for any $\rho \in (0,1)$,
\begin{align*}
\textstyle \mathbb P\!\left(
\|\bar r_t\|_2
\le
\sqrt{\frac{n}{(1-\rho)d}}\,\|U_t^\top \bar r_t\|_2
\quad \forall\, \bar r_t \in \operatorname{span}(V)
\;\middle|\; \mathcal F_t
\right)
\\
\textstyle \ge
1
-
\frac{1}{p_{\mathrm{acc}}}
\left(\frac{12n}{\rho d}\right)^m
\exp\!\left(-\frac{\rho^2 d}{32}\right).
\end{align*}
Equivalently, if
$d
\ge
\frac{32}{\rho^2}
\left[
m\log\!\left(\frac{12n}{\rho m}\right)
+
\log\!\left(\frac{1}{\delta\,p_{\mathrm{acc}}}\right)
\right],$
then
\begin{align*}
\textstyle
\mathbb P\!\left(
\|\bar r_t\|_2
\!\le\!
\sqrt{\frac{n}{(1\!-\!\rho)d}}\,\|U_t^\top \bar r_t\|_2, \forall\bar r_t \!\in\! \operatorname{span}(V)
\middle| \mathcal F_t
\right)
\!\ge\!
1\!-\!\delta.
\end{align*}
\end{lemma}
\begin{proof}
Let
\begin{align*}
A_t := V^\top U_t U_t^\top V \in \mathbb R^{m\times m}.
\end{align*}
It suffices to show that
$\lambda_{\min}(A_t) \ge \frac{(1-\rho)d}{n}.
$

Set $\varepsilon := \frac{\rho d}{4n}.$
Let $\mathcal N_\varepsilon$ be an $\varepsilon$-net of the unit sphere $\mathbb S^{m-1}$ satisfying
\vspace{-10pt}
\begin{align*}
\textstyle|\mathcal N_\varepsilon|
\le
\left(\frac{3}{\varepsilon}\right)^m
=
\left(\frac{12n}{\rho d}\right)^m.
\end{align*}
For any $z \in \mathcal N_\varepsilon$, the vector $Vz$ is $\mathcal F_t$-measurable and satisfies $\|Vz\|_2=1$. Therefore, applying Lemma~\ref{lem:conditional_projection} with parameter $\rho/2$ gives
\begin{align*}
\mathbb P\!\left(
z^\top A_t z
=
\|U_t^\top Vz\|_2^2
\ge
\frac{(1-\rho/2)d}{n}
\;\middle|\; \mathcal F_t
\right)
\\\ge
1-\frac{1}{p_{\mathrm{acc}}}\exp\!\left(-\frac{\rho^2 d}{32}\right).
\end{align*}
A union bound over $z \in \mathcal N_\varepsilon$ yields
\begin{align*}
\mathbb P\!\left(
z^\top A_t z \ge \frac{(1-\rho/2)d}{n}
\quad \forall\, z \in \mathcal N_\varepsilon
\;\middle|\; \mathcal F_t
\right)
\\\ge
1
-
\frac{1}{p_{\mathrm{acc}}}
\left(\frac{12n}{\rho d}\right)^m
\exp\!\left(-\frac{\rho^2 d}{32}\right).
\end{align*}

It remains to pass from the net to the whole sphere. Fix any $x \in \mathbb S^{m-1}$ and choose $z \in \mathcal N_\varepsilon$ such that $\|x-z\|_2 \le \varepsilon$. Then
\begin{align*}
\bigl|x^\top A_t x - z^\top A_t z\bigr|
\le
\|x-z\|_2\,\|A_t x\|_2 + \|z\|_2\,\|A_t(x-z)\|_2
\le
2\varepsilon.
\end{align*}
Therefore, on the preceding net event,
\begin{align*}
x^\top A_t x
\ge
z^\top A_t z - 2\varepsilon
\ge
\frac{(1-\rho/2)d}{n} - \frac{\rho d}{2n}
=
\frac{(1-\rho)d}{n}.
\end{align*}
Since this holds for every $x \in \mathbb S^{m-1}$, we obtain
\begin{align*}
\textstyle \lambda_{\min}(A_t)
=
\inf_{\|x\|_2=1} x^\top A_t x
\ge
\frac{(1-\rho)d}{n},
\end{align*}
which is exactly the claim.
\end{proof}
\begin{lemma}[Random projection under rejection sampling]
\label{lem:conditional_projection}
Let $\bar r_t$ be $\mathcal F_t$-measurable, and let $U_t$ be obtained by rejection sampling from Haar-uniform draws on $\mathrm{St}(n,d)$ until $\mathcal A_t$ occurs, with acceptance rate at least $p_\mathrm{acc}$. Then for any $\rho \in (0,1)$,
\begin{align*}
\textstyle \mathbb P\!\left(
\|\bar r_t\|_2
\le
\sqrt{\frac{n}{(1-\rho)d}} \, \|U_t^\top \bar r_t\|_2
\;\middle|\; \mathcal F_t
\right)
\ge
1 - \frac{1}{p_{\mathrm{acc}}}\exp\!\left(-\frac{\rho^2 d}{8}\right).
\end{align*}
\end{lemma}

\begin{proof}
Conditioned on $\mathcal F_t$, each proposal $\widetilde U_t$ is Haar-uniform. The accepted $U_t$ has the law of $\widetilde U_t$ conditioned on $\mathcal A_t$. Thus for any event $\mathcal B_t$,
\begin{align*}
\textstyle \mathbb P(\mathcal B_t \mid \mathcal F_t, \mathcal A_t)
=
\frac{\mathbb P(\mathcal B_t \cap \mathcal A_t \mid \mathcal F_t)}
{\mathbb P(\mathcal A_t \mid \mathcal F_t)}.
\end{align*}
Let $\mathcal B_t$ be the standard projection concentration event. Then (cf. Lemma 2.2 in \cite{dasgupta2003elementary})
\begin{align*}
\textstyle \mathbb P(\mathcal B_t^c \mid \mathcal F_t)
\le
\exp\!\left(-\frac{\rho^2 d}{8}\right),
\end{align*}
and dividing by $\mathbb P(\mathcal A_t \mid \mathcal F_t) \ge p_{\mathrm{acc}}$ gives the result.
\end{proof}
\end{document}